\documentclass[10pt,A4,conference]{IEEEtran}
 \usepackage[left=1.62cm,right=1.62cm,top=1.9cm]{geometry}
 \usepackage{authblk}
\usepackage{tikz}  
\usetikzlibrary{calc,intersections}
 \usepackage{verbatim}
\usepackage{textcomp}
\usepackage[T1]{fontenc}
\usepackage{bbm}
\usepackage{multirow}
\usepackage{dsfont}
\usepackage{cite}
\usepackage{epsfig}
\usepackage{float}
\usepackage{enumerate}
\usepackage{balance}
\usepackage{color}
\usepackage{subcaption}
\usepackage{caption}
\usepackage{algorithm}
\usepackage[noend]{algpseudocode}
\algrenewcommand\Return{\State \algorithmicreturn{} } 
\usepackage{hyperref}
\hypersetup{
    colorlinks=true,
    linkcolor=blue,
    filecolor=magenta,      
    urlcolor=cyan,
		citecolor=blue,
}
\usepackage{amssymb}
\usepackage{amsthm}
\usepackage{array}
\usepackage{graphicx}
\usepackage{lettrine}
\usepackage{epsf} 
\usepackage{psfrag}
\usepackage[usenames,dvipsnames]{pstricks}
\usepackage{pst-grad} 
\usepackage{pst-plot} 
\usepackage[space]{grffile} 
\usepackage{etoolbox} 
\usepackage{url} 

\usepackage{multirow}

\newtheorem{theorem4}{Theorem}
\newtheorem{lemma}[theorem4]{Lemma}

\usepackage{optidef}
\usepackage[colorinlistoftodos,bordercolor=orange,backgroundcolor=orange!20,linecolor=orange,textsize=scriptsize]{todonotes}

\begin{document}

\title{Age-of-Updates Optimization for UAV-assisted Networks}

 
\author[1]{Mouhamed Naby Ndiaye}
\author[1]{El Houcine Bergou}
\author[2]{Mounir Ghogho}
\author[1]{Hajar El Hammouti}
\affil[1]{School of Computer Science, Mohammed VI Polytechnic University (UM6P), Benguerir, Morocco, \authorcr
emails: {\{naby.ndiaye,hajar.elhammouti,elhoucine.bergou\}@um6p.ma}}
\affil[2]{TICLab, International University of Rabat (UIR), Rabat, Morocco,\authorcr
email: {mounir.ghogho@uir.ac.ma}}


\maketitle

\begin{abstract}

Unmanned aerial vehicles (UAVs) have been proposed as a promising technology to collect data from IoT devices and relay it to the network. In this work, we are interested in scenarios where the data is updated periodically, and the collected updates are time-sensitive. In particular, the data updates may lose their value if they are not collected and analyzed timely. To maximize the data freshness, we optimize a new performance metric, namely the Age-of-Updates (AoU). Our objective is to carefully schedule the UAVs hovering positions and the users' association so that the AoU is minimized. Unlike existing works where the association parameters are considered as binary variables, we assume that devices send their updates according to a probability distribution. As a consequence, instead of optimizing a deterministic
objective function, the objective function is replaced by an
expectation over the probability distribution. The expected AoU is therefore optimized under quality of service and energy constraints. The original problem being non-convex, we propose an equivalent convex optimization that we solve using an interior-point method. Our simulation results show the performance of the proposed approach against a binary association.


\end{abstract}
\IEEEoverridecommandlockouts

\IEEEpeerreviewmaketitle


\section{Introduction}
In the last few years, unmanned aerial vehicles (UAVs) have been considered for a variety of applications in the telecommunications industry~\cite{FAAFiscal}. In particular, UAVs are used as base stations to extend the network coverage in dense and out-of-reach areas~\cite{10.1007/978-3-030-02849-7_33}. They are also deployed as data collectors to relay information from Internet of Things (IoT) devices to sink nodes~\cite{alzenad20173}. In this work, we address the problem of UAVs acting as data collectors. We are interested in the context where the collected data is updated periodically and the data updates are time-sensitive~\cite{wang2020priority}. In particular, the data updates may lose their meaning and value if they are not collected and analyzed timely. As a consequence, it is important to carefully design the UAV's flight and stopping points so that the collected updates are kept as fresh as possible. For this purpose, a new performance metric, namely Age-of-Updates (AoU), has been introduced. This metric captures the time since the last update was collected. In this paper, we answer the question: How to schedule the hovering locations of the drones and determine the selected devices, from which the data is collected, so that the total AoU is minimized? 


\subsection{Related Work}
The majority of UAV-enabled networks research works focus on optimizing metrics such as the sum-rate and energy efficiency~\cite{zeng2017energy,Zeng2019,zhan2019completion,licea2019communication,liu2018energy,ranjha2021urllc}. For example, in~\cite{licea2019communication}, the authors design the trajectory of a UAV to maximize the rate between a drone and a ground base station. The authors consider the limited energy of the drone, and the dynamic nature of the communication channel. They decompose the problem into a sequence of control optimizations that are solved using control theory. Another example can be found in~\cite{liu2018energy} where the authors propose an energy-efficient drone control policy to ensure efficient and fair communication coverage for IoT devices using deep reinforcement learning.
 
However, only a handful of papers address the problem of data timeliness in UAV-enabled networks. In~\cite{jia2019age}, the authors optimize a similar concept, called Age-of-Information (AoI). AoI is defined as the time since the last relevant information arrived at its destination. The objective of the work in~\cite{jia2019age} is to optimize the data collection mode and the UAV trajectory to reduce the average AoI of IoT devices. In the same context, in~\cite{tong2019uav}, the authors propose a UAV trajectory planning to minimize the maximum AoI of a UAV-enabled wireless sensor network. Another work in~\cite{abd2018average} consists of optimizing the service time allocation and the flight of a UAV used as a mobile relay. The previously cited works are limited to a single drone application and optimize only the 2D location of the drone, ignoring the impact of its altitude. 

The AoU has been introduced recently in the context of federated learning (FL)~\cite{yang2020age}. An age measure of this type has been shown to improve the notion of data freshness in a variety of applications and has already been used, primarily in the context of networking (see, for example, the study in~\cite{yates2021age}). In~\cite{yang2020age}, the authors propose a device scheduling policy to maximize the parameters' updates freshness, and therefore, improve the accuracy of the FL model. The studied problem is combinatorial. It involves mixed-integer and continuous variables and it is solved using a decomposition method that addresses the channel allocation and devices association subproblems separately. 

In fact, UAV networks related optimizations are often formulated as mixed-integer non-linear programming, where the scheduling/association parameters are considered as binary variables, whereas the positions of UAVs are modelled as continuous variables. As a result, the optimization problems are non-convex, most of the time NP-hard, and therefore, difficult to solve. Moreover, the proposed algorithms are heuristics with no performance guarantee and, in general, require a long time to converge.

In our paper, we deal with the problem of AoU optimization in UAV-assisted networks differently. We claim that, the binary association constraint can be relaxed. The association variables can be seen as probabilities. In particular, a user decides whether to connect to a UAV or not based on a probability distribution. As a consequence, instead of optimizing a deterministic objective function, the objective function is replaced by an expectation over the probability distribution. This approach has two main advantages. First, it allows to formulate the problem as a continuous optimization which can be solved numerically with some guarantees on the final solution. Such an approach can even result in continuous and convex problems which can be solved efficiently using polynomial-time algorithms such as interior-point methods  \cite{boyd2004convex}. Second, since the association is  probabilistic (i.e., a device decides to send its updates or not by randomly drawing samples from a probability distribution), it allows the participation of a large number of devices, which leads to some fairness among devices.

\subsection{Contribution}
In this paper, we consider a UAV-assisted network, where UAVs act as relays between IoT devices and a macro base station (BS). Due to the limited communication and energy resources, we assume that only a subset of IoT devices send their updates to the UAVs. Similarly, only a subset of UAVs interact with the BS to send the collected updates. The main objective of this work is to schedule the UAVs' hovering locations and optimize the association probabilities (i.e., the probabilities between IoT devices and UAVs, and between UAVs and the BS) so that the collected updates are kept as fresh as possible. The contributions of our paper are summarized as follows.
\begin{itemize}
    \item We formulate the joint association probabilities and 3D locations problem as a non-linear programming (NLP) optimization where the objective is to minimize the expected AoU under energy and quality of service constraints.
    
    \item To solve the underlying optimization efficiently, we propose a reformulation of the problem. We show that the new formulation of the problem is convex. Hence, we propose an interior-point based
algorithm which provides a near-optimal solution to the studied optimization.

\item Finally, we validate our proposed approach with simulation experiments. In particular, we show that the proposed joint probability association and 3D position scheduling outperforms a benchmark that adopts binary association. 

\end{itemize}



\subsection{Organization}
The remainder of the paper is organized as follows. First, the studied system model is described in Section~\ref{Sys}. The mathematical formulation of the problem is given in Section~\ref{Prob}. In Section~\ref{ResM}, we propose a convex reformulation of the problem. In Section~\ref{Algo}, we describe the joint association probability and 3D locations to solve the underlying optimization efficiently. Next, in Section~\ref{Simu}, we show the performance of our algorithm for various scenarios. Finally, concluding remarks are provided in Section~\ref{Conc}.


\section{System Model}\label{Sys}
\subsection{Communication Model}
We consider a set $\mathcal{I}$ of $I$ IoT devices within an area $\mathcal{A}$. The IoT devices collect data from their surrounding environment and send data updates periodically to a server located at the BS. Since IoT devices have a limited communication range, a set $\mathcal{U}$ of UAVs are deployed to collect data on a regular basis from IoT devices and re-transmit it to the BS. An example of the studied system model is illustrated in Fig.~\ref{fig:systemmodel}. 

\begin{figure}[ht]
    \centering
    \includegraphics[width=1\linewidth]{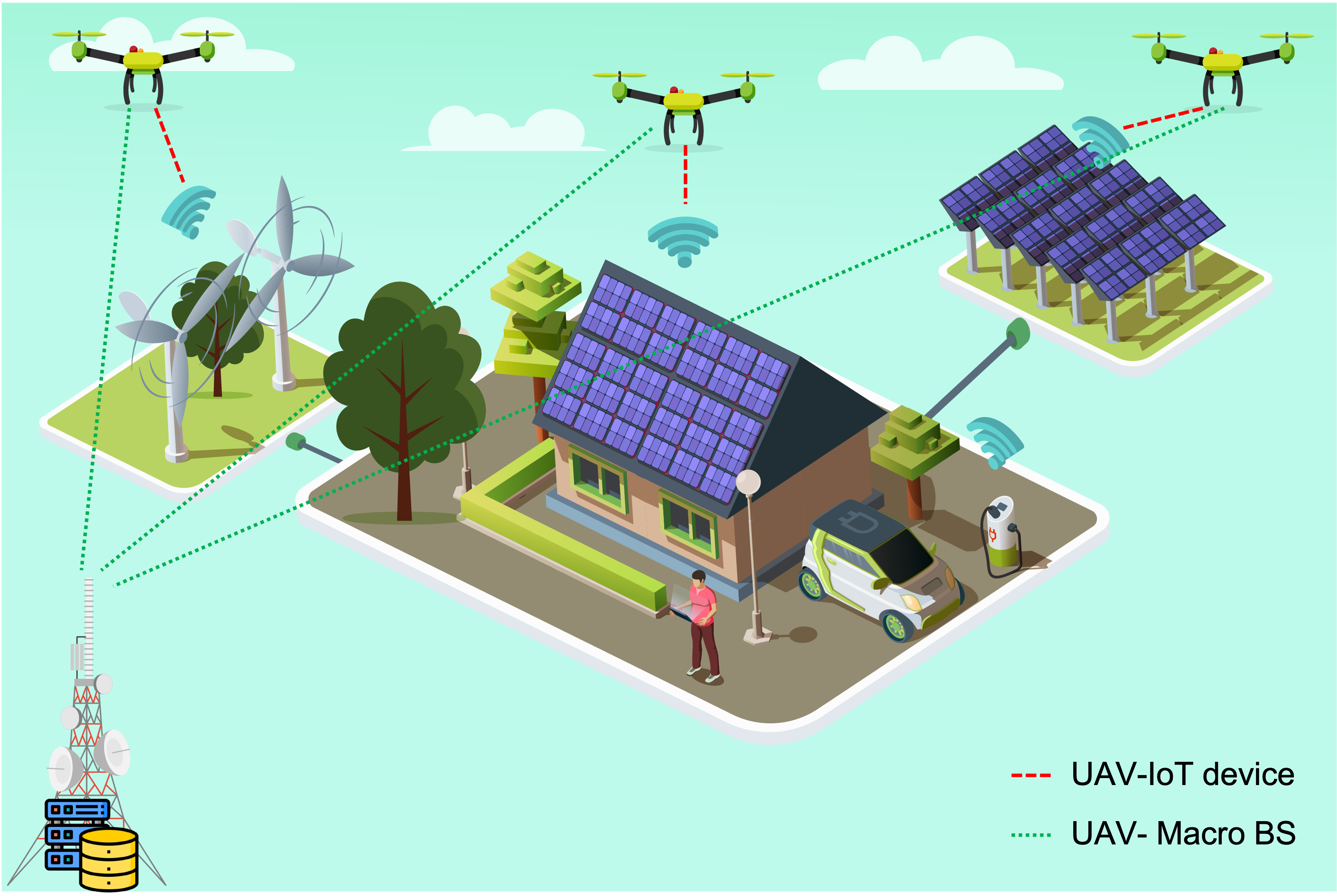}
    \caption{System Model}
    \label{fig:systemmodel}
\end{figure}

We suppose that the IoT devices generate data updates that are transmitted periodically. At the end of each time interval $k$ of duration $\tau$ (i.e., $\left[k\tau,(k+1)\tau \right]$, with $k \in \mathcal{K}= \{0,1,\dots,K\}$, and $KT$ the time horizon),
an IoT device $i$ sends its data updates to UAV $u$ with some probability, which in turn transmits the updates to the macro BS in the following time interval $k+1$ with another probability. Each UAV communicates with both IoT devices and the BS via the air-to-ground channel. To model the air-to-ground channel, we assume a Rician fading distribution $\widehat{\Delta}_{i,u}[k]$

\begin{equation}
    \widehat{\Delta}_{i,u}[k]=\left(\sqrt{\frac{\Omega}{\Omega+1}} \bar{\Delta}_{i,u}[k]+\sqrt{\frac{1}{\Omega+1}} \widetilde{\Delta}_{i,u}[k]\right),
\end{equation}
where $\Omega$ is the Rician factor, $\bar{\Delta}_{i,u}[k]$ is the line-of-sight (LoS) component with $\left|\bar{\Delta}_{i,u}[k]\right|=1$, and $\widetilde{\Delta}_{i,u}[k]$ the random non-line-of-sight (NLoS) component where $\widetilde{\Delta}_{i,u}[k] \sim \mathcal{C N}(0,1)$ where $\mathcal{C N}(0,1)$ is the complex normal distribution.

Each IoT device $i$ transmits its updates with a power $P_i^D[k]$ during time interval $k$. Thus, the received power at UAV $u$ is $\left| \widehat{\Delta}_{i,u}[k]\right|^2\beta_{0}\left(d_{i,u}[k]\right)^{-2}P_i^D[k]$, where  $\beta_{0}$ is the average channel power gain at a reference distance $d_{0}=1 \mathrm{~m}$, and $d_{i,u}[k]$ is the distance between device $i$ and UAV $u$ during time slot $k$,  
\begin{equation}
    d_{i,u}[k]=\left\|\textbf{q}_{i}- \textbf{w}_{u}[k] \right\|_2, \quad k=1,2, \ldots ,K,
\end{equation}
where $\left\|.\right\|_2$ is the two-norm, $\textbf{q}_{i}=\left[ x_{i}, y_{i},0 \right]$ is the position vector of IoT device $i$, and $\textbf{w}_{u}[k]=\left[ x_{u}[k], y_{u}[k],h_{u}[k] \right]$ is the 3D location vector of UAV $u$ during time interval $k$.




The signal-to-noise ratio (SNR) of IoT device $i$ with respect to UAV $u$ is given by
 \begin{equation}
     \Upsilon_{i,u}^D[k]= P_{i}^D[k]\left|\widehat{\Delta}_{i,u}[k]\right|^{2} \beta_{0}\left(d_{i,u}[k]\right)^{-2} / \sigma^{2}, 
 \end{equation}
 where $\sigma^{2}$ is the thermal noise power. Therefore, the rate of IoT device $i$ when it is associated with UAV $u$ during time slot $k$ can be written
 
 \begin{equation}
     R_{i,u}^D[k]=\Lambda_{i,u}^D[k]\log _{2}\left(1+\Upsilon_{i,u}^D[k]\right),
 \end{equation}
where $\Lambda_{i,u}^D[k]$ is the allocated bandwidth between device $i$ and UAV $u$ during time slot $k$.

Similarly, the rate of UAV $u$ when it transmits the data updates to the BS during time slot $k$ is given by
 \begin{equation}
     R_{u}^U[k]=\Lambda_u^U[k]\log _{2}\left(1+\Upsilon_{u}^U[k]\right),
 \end{equation}
where $\Lambda_u^U[k]$ is the bandwidth of UAV $u$ during time slot $k$, and $\Upsilon_{u}^U[k]$ is the SNR of UAV $u$ when it transmits to the BS during time slot $k$ which is given by

 \begin{equation}
     \Upsilon_{u}^U[k]= P_{u}^U[k]\left|\widehat{\Delta}_{u}[k]\right|^{2} \beta_{0}\left(d_{u}[k]\right)^{-2} / \sigma^{2}, 
 \end{equation}
with $P_{u}^U[k]$ the transmit power of UAV $u$ and  $d_{u}[k]$ the distance between UAV $u$ and the BS during time interval $k$.
\subsection{Energy Consumption Model}
Due to their limited battery budget, it is important to account for the energy that a UAV consumes during its mission. In our work, we account for three types of energy. First, the flying energy, which is the energy to travel from one location to another. Second, the hovering energy, which is the energy that the UAV remains aloft and supports its movement. And finally, the communication energy which is the transmission energy. Accordingly, the total consumed energy by a UAV $u$ during time interval $k$ is given by

\begin{align}
    E_{u}[k]&=E_{u}^{f}[k]+E_{u}^{h}[k]+E_u^c[k]\nonumber\\
    &=P_{u}^{f} \tau_{u}^{f}[k] + (P_{u}^{h} + P_{u}^{U}[k])(\tau-\tau_{u}^{f}[k]),
    \end{align}
where $P_{u}^{f}$ and $P_{u}^{h}$ are the propulsion power of a rotary-wing UAV and the hovering power respectively, that we assume constant over time. $\tau_{u}^{f}[k]$ is the time the UAV $u$ makes to travel from one location to another during time slot $k$, it is given by  $\tau_{u}^{f}[k]=v^{-1} \left\| \textbf{w}_{u}[k]-\textbf{w}_{u}[k+1] \right\|_2$, where $v$ is the speed of the UAV that we assume constant. 


\subsection{Age-of-Updates Metric}
The objective of this work is to maximize the total data collected, while maximizing the freshness of the data and ensure that its UAVs complete their misson. In particular, our aim is to minimize the AoU.

Let $A_{i,u}$ be the probabilistic event that IoT device $i$ is associated with UAV $u$. We denote by $a_{i,u}$ the probability that $A_{i,u}=1$. Also, let $B_{u}$ be the probabilistic event that the UAV $u$ is selected by the BS and $b_{u}$ is the probability that $B_{u}=1$. More formally, we have:

\begin{equation*}
A_{i, u}=\left\{\begin{array}{lc}
1, & \text { with probability } a_{i,u}  \\
0, & \text { with probability } 1-a_{i,u}
\end{array}\right.
\end{equation*}
\begin{equation*}
B_{u}=\left\{\begin{array}{lc}
1, & \text { with probability } b_{u}  \\
0, & \text { with probability } 1-b_{u}
\end{array}\right.
\end{equation*}

To model the updates' freshness, we use the same definition as in \cite{Yang2019}. Particularly, if device $i$ is associated with UAV $u$ during time interval $k$, its AoU evolves as follows,

\begin{equation}
    T_{i,u}[k+1]=\left(T_{i,u}[k]+1\right)\left(1-A_{i,u}[k]\right),
    \label{Tiu}
\end{equation}
where $T_{i,u}[0]=0$. $T_{i,u}[k]$ is the age of the updates received by drone $u$ and collected from device $i$ during time interval $k$. Specifically, when the updates of device $i$ are not collected during time interval $k$ (i.e., $A_{i,u}[k]=0$), the AoU is increased by one unit of time. Inversely, when the updates are transmitted, the AoU is reinitialized to zero.


Therefore, the AoU of IoT devices whose data transit through UAV $u$ and is received by the BS is given by
\begin{equation}
    T_{u}[k+1]=\left(\sum_{ i \in \mathcal{I}} T_{i,u}[k-1]  A_{i,u}[k-1]+1\right)\left(1-B_{u}[k]\right).
\end{equation}



Finally, the global AoU at the BS, during time interval $k+1$, with $k>1$, can be written
\begin{equation}
{AoU}[k+1]=\!\!\sum_{  u\in \mathcal{U}}\!\!\left(\!\!\sum_{ i \in \mathcal{I}}\left( T_{i,u}[k-1] A_{i,u}[k-1]+1\right) \left(1-B_{u}[k]\right)\!\!\right)
\label{GlobalAOU}
\end{equation}

Due to the random behavior of events $A_{i,u}$ and $B_u$, we target to minimize the expected value of the global AoU as perceived by the BS. In Lemma~\ref{lemma1}, we give the expression of the studied metric.
\begin{figure*}
\hrule
\begin{equation}
\begin{array}{rcl}
& \mathbb{E}(AoU[k+1])=  \mathbb{E}\Bigg[ \sum_{  u\in \mathcal{U}}\sum_{ i \in \mathcal{I}}\bigg(\left( T_{i,u}[k-1] A_{i,u}[k-1]+1\right) \left(1-B_{u}[k]\right)\bigg)\Bigg]\\
& =  \mathbb{E}\Bigg[ \sum_{  u\in \mathcal{U}}\sum_{ i \in \mathcal{I}}\bigg(T_{i,u}[k-1] A_{i,u}[k-1] +(1 - B_{u}[k])  - T_{i,u}[k-1] A_{i,u}[k-1] B_{u}[k]\bigg) \Bigg]\\
& =  \sum_{  u\in \mathcal{U}}\sum_{ i \in \mathcal{I}}\Bigg(\mathbb{E}\bigg[T_{i,u}[k-1] A_{i,u}[k-1]\bigg] +\mathbb{E}\bigg[(1 - B_{u}[k])\bigg]  -\mathbb{E}\bigg[T_{i,u}[k-1] A_{i,u}[k-1] B_{u}[k]\bigg]\Bigg)  \\
& =  \sum_{  u\in \mathcal{U}}\sum_{ i \in \mathcal{I}}\Bigg(T_{i,u}[k-1] \mathbb{E}\Big[A_{i,u}[k-1]\Big] +\bigg(1 - \mathbb{E}\Big[B_{u}[k]\Big]\bigg) - T_{i,u}[k-1] \mathbb{E}\Big[A_{i,u}[k-1]\Big] \mathbb{E}\Big[B_{u}[k]\Big]\Bigg)  \\
& =  \sum_{  u\in \mathcal{U}} \sum_{  i\in \mathcal{I}}\left(T_{i,u}[k-1] a_{i,u}[k-1] +(1 - b_{u}[k])  -T_{i,u}[k-1] a_{i,u}[k-1] b_{u}[k]\right)
\end{array}\\
\label{prooflemma1}
\end{equation}
\end{figure*}
\begin{lemma}
Under the assumption that $A_{i,u}[k]$ and $B_{u}[k]$ are independent for all $k$, $u$ and $i$, the expected value of the global AoU, during time interval $k+1$ is given by
\begin{equation}
\begin{aligned}
& \mathbb{E}(AoU[k+1])= \sum_{  u\in \mathcal{U}} \sum_{  i\in \mathcal{I}}\left(T_{i,u}[k-1] a_{i,u}[k-1]+(1 - b_{u}[k])  \right.\\
&\left.-T_{i,u}[k-1] a_{i,u}[k-1] b_{u}[k]\right).
\end{aligned}
\label{expectedAoU}
\end{equation}
\label{lemma1}
\end{lemma}
\begin{proof}
This expression is obtained by developing the AoU and applying the properties of the expectation, as shown in equation~ (\ref{prooflemma1}).

\end{proof}

Next, we formulate the expected AoU optimization as a NLP. Our target is to find the optimal probability association and 3D positioning of UAVs under energy and quality of service constraints.
\section{Problem Formulation }\label{Prob}
In this section, we first describe the constraints of the studied system. Then, we propose a mathematical formulation of the problem.

First, to ensure that the updates are transmitted timely, it is required that the expected uplink rate is above a predefined threshold. Let $\bar{R_I}$ and $\bar{R_U}$ be the rate thresholds for IoT devices and UAVs, respectively. As a consequence, the quality of service constraints are expressed as follows

\begin{equation}
 a_{i,u}[k]R_{i,u}^D[k]  \ge  \bar{R_{I}},\quad \forall i \in \mathcal{I}, \forall k \in \mathcal{K}
  \label{third:b}
\end{equation}

\begin{equation}
b_{u}[k] R_{u}^U[k]  \ge  \bar{R_{U}} ,\quad \forall u \in \mathcal{U}, \forall k \in \mathcal{K}. 
  \label{third:c}
\end{equation}
Additionally, UAVs must be able to accomplish their mission without exceeding their energy budget $E_u^{\rm max}$. Therefore, the expected consumed energy should satisfy at each time slot the following constraint 
\begin{equation}
 b_{u}[k] E_{u}[k] \le  E_{u}^{\text{max}} , \; \forall u \in \mathcal{U}, \forall k \in \mathcal{K}.
\label{third:f}
\end{equation}
Moreover, to avoid any collision between UAVs, we set a minimum distance $d_{min}$ between any two UAVs, thus

\begin{equation}\label{dist}
    \left\| \textbf{w}_{u}[k]-\textbf{w}_{v}[k] \right\|_2\ge d_{min}\  \forall u \in \mathcal{U},v \in \mathcal{U}, u\neq v, \forall k \in \mathcal{K}. 
\end{equation}
Finally, our target is to optimize the expected AoU over time. Therefore, the optimization problem is formulated as follows. 
\begin{subequations}\label{prob:main}
\begin{align}
\begin{split}
\begin{aligned}
& \mathcal{P}:\min_{\mathbf{a},\mathbf{b},\mathbf{w}}\sum_{k \in \mathcal{K}}\sum_{  u\in \mathcal{U}} \sum_{  i\in \mathcal{I}}\left(T_{i,u}[k-1] a_{i,u}[k-1] +(1 - b_{u}[k])  -\right.\\
&\left.T_{i,u}[k-1] a_{i,u}[k-1] b_{u}[k]\right)
\end{aligned}
\label{third:a}
\end{split}\\
\begin{split}
 \text{s . t.} \quad (\ref{third:b}) - (\ref{dist})
\end{split}\\
\begin{split}
 \sum_{u \in \mathcal{U}}a_{i,u}[k]\le 1, \; \forall i \in \mathcal{I}, \forall k \in \mathcal{K}
  \label{third:e}
\end{split}\\
\begin{split}
0 \le a_{i,u}[k] \le 1, \forall i \in \mathcal{I}, u \in \mathcal{U}, \forall k \in \mathcal{K}
  \label{third:h}
\end{split}\\
\begin{split}
 0 \le b_{u}[k] \le 1, \forall u \in \mathcal{U}, \forall k \in \mathcal{K}
  \label{third:i}
\end{split}\\
\begin{split}
     x_{u\text{min}} \le x_{u}[k] \le x_{u\text{max}}, \forall u \in \mathcal{U}, \forall k \in \mathcal{K}
  \label{third:l}
\end{split}\\
\begin{split}
    y_{u\text{min}} \le y_{u}[k] \le y_{u\text{max}}, \forall u \in \mathcal{U}, \forall k \in \mathcal{K}
  \label{third:m}
\end{split}\\
\begin{split}
    z_{u\text{min}}\le z_{u}[k] \le z_{u\text{max}},\forall u \in \mathcal{U}, \forall k \in \mathcal{K}
  \label{third:n}
\end{split}
\end{align}
\label{PF}
\end{subequations}
where $\mathbf{a}$ and $\mathbf{b}$ are the association probability vectors and $\mathbf{w}$ is the 3D position vector over time. In this formulation, the constraint (\ref{third:e}) ensures that the association probabilities over UAVs do not exceed $1$. The constraints (\ref{third:h}) and (\ref{third:i}) ensure that the probability vectors $\mathbf{a}$ and $\mathbf{b}$ are well-defined. Finally, the constraints (\ref{third:l}), (\ref{third:m}) and (\ref{third:n}) restrict the movement of the UAVs to a limited $3D$ space. 

The Problem $\mathcal{P}$ includes only  continuous variables, which makes it a continuous non-linear and non-convex programming problem. It is not convex, because constraints (\ref{third:b}), (\ref{third:c}) and the objective function are not convex.  
In the next section, we propose a convex reformulation of Problem $\mathcal{P}$ that can be solved efficiently using methods of convex programming~\cite{boyd2004convex}.

\section{A Convex Reformulation of AoU Optimization}\label{ResM} 


In this section, 
 we propose a new convex reformulation of Problem $\mathcal{P}$. For all $k$, to deal with the non-convexity factor $a_{i,u}[k-1]  b_{u}[k]$ in the objective function, we introduce a  slack variable $t_k=a_{i,u}[k-1] b_{u}[k]$.  
 Since $t_k$ is the multiplication of two probabilities, it should be between $0$ and $1$. The objective function using the new variables then becomes  
 \begin{equation*}
\begin{aligned}
& \sum_{k \in \mathcal{K}} \sum_{  u\in \mathcal{U}} \sum_{  i\in \mathcal{I}}\left(T_{i,u}[k-1] a_{i,u}[k-1] +(1 - b_{u}[k])  -\right.\\
&\left.T_{i,u}[k-1] t_k\right),
\end{aligned}
\end{equation*}
 which is a linear function of the decision variables. The following lemma deals with the non convexity of the constraint (\ref{third:b}). The constraint (\ref{third:c}) is similar to (\ref{third:b}), thus the same reformulation applies. 

 \begin{lemma}
For each $k\in \mathcal{K}$, the IoT device's rate constraint (\ref{third:b}), can be rewritten as 
\begin{equation}
   \frac{e^{\left(\frac{\bar{R_{I}}}{a_{i,u}[k]\Lambda_{i,u}^D}\right)}-1}{P_{i}^D/\sigma^{-2}\left|\widehat{\Delta}_{i,u}[k]\right|^{2} \beta_{0}} - t_{1} \le 0, \text{ where }
    \label{conv_rate_Iot}
\end{equation}
 \begin{equation}
    \frac{1}{(x_{u}[k]-x_{i})^{2}+(y_{u}[k]-y_{i})^{2}+(z_{u}[k])^{2}}-t_{1}=0.
    \label{t1}
\end{equation}
 The functions defining (\ref{conv_rate_Iot}) and (\ref{t1}) are both convex.
\label{prop1}
\end{lemma}
\begin{proof}
After applying the exponential function, the constraint (\ref{third:b}) becomes 
$$
\Upsilon_{i,u}^D[k]\geq e^{\left(\frac{\bar{R_{I}}}{b_{u}[k]\Lambda_{i,u}^D}\right)}-1,
$$
where $\Upsilon_{i,u}^D[k]=P_i^D[k]\left|\widehat{\Delta}_{i,u}[k]\right|^{2} \beta_{0}\sigma^{-2}\left(d_{i,u}[k]\right)^{-2}$.
Now, by letting 
$
    t_{1}=\frac{1}{(x_{u}[k]-x_{i})^{2}+(y_{u}[k]-y_{i})^{2}+(z_{u}[k])^{2}},
$
we get 
$$
  \frac{e^{\left(\frac{\bar{R_{I}}}{a_{i,u}[k]\Lambda_{i,u}^D}\right)}-1}{P_i^D[k]/\sigma^{-2}\left|\widehat{\Delta}_{i,u}[k]\right|^{2} \beta_{0}} - t_{1} \le 0.
$$
Finally, note that the functions $h_1(x)=e^{\frac{1}{x}}$, and $h_2(x)=\frac{1}{x^2}$ are both convex.

\end{proof}

Note all the remaining constraints in the definition of Problem $\mathcal{P}$ are convex.

Now that we showed the convex reformulation of  Problem $\mathcal{P}$. In the next section, we outline the main steps of our proposed approach to perform the Joint probability Association and 3D position Scheduling (JAS) of IoT devices, UAVs and the BS.

\section{Joint probability Association and position Scheduling for AoU Minimization}\label{Algo}
\begin{figure*}[t]
\centering
\minipage{0.24\textwidth}
  \includegraphics[width=1\linewidth]{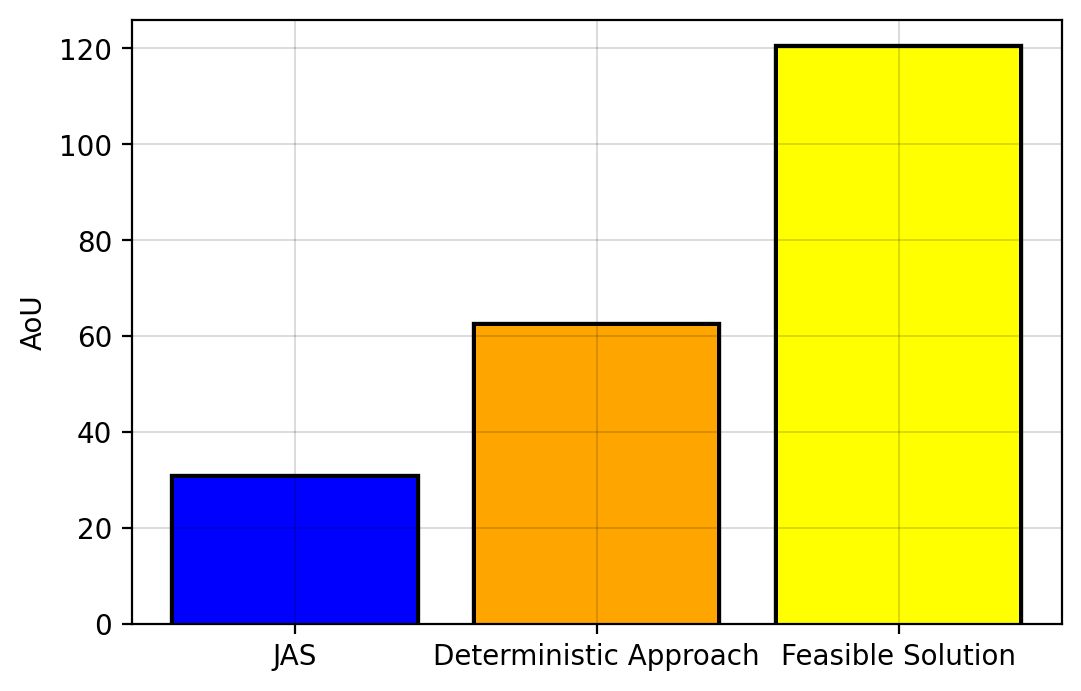}
    \caption{Comparison between \\ JAS and Deterministic approach}
    \label{fig:comp}
\endminipage\hfill
\minipage{0.24\textwidth}
  \includegraphics[width=1\linewidth]{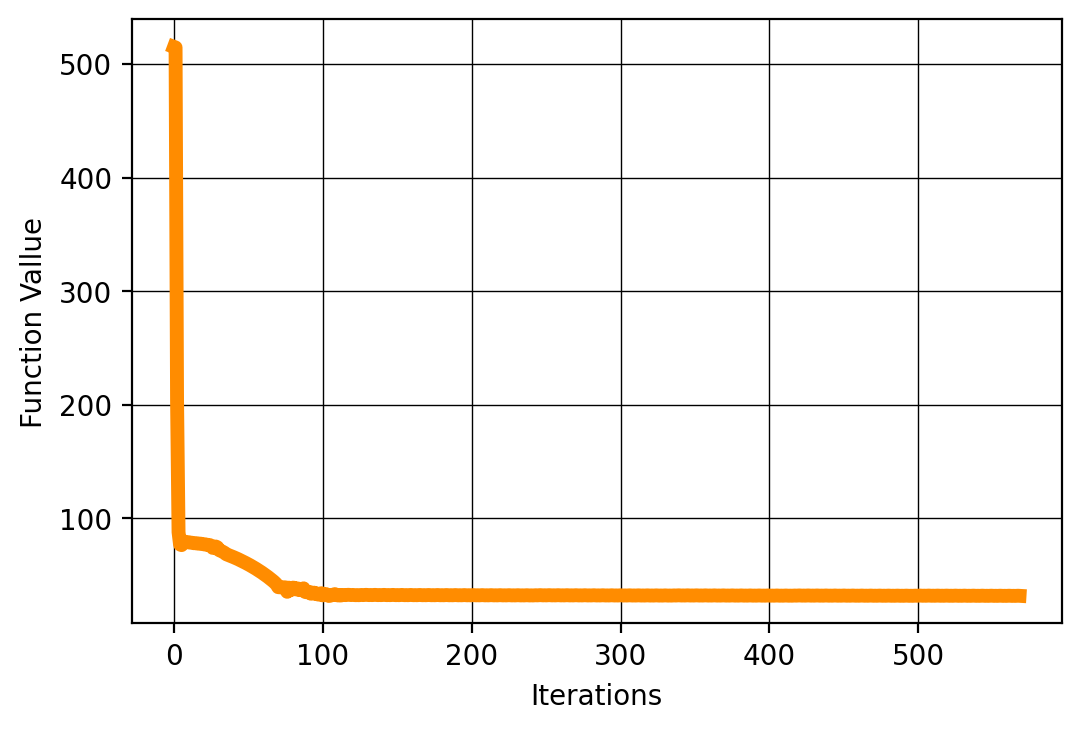}
    \caption{Expected AoU over iterations}
    \label{fig:funcplot}
\endminipage \hfill
\minipage{0.24\textwidth}
  \includegraphics[width=1\linewidth]{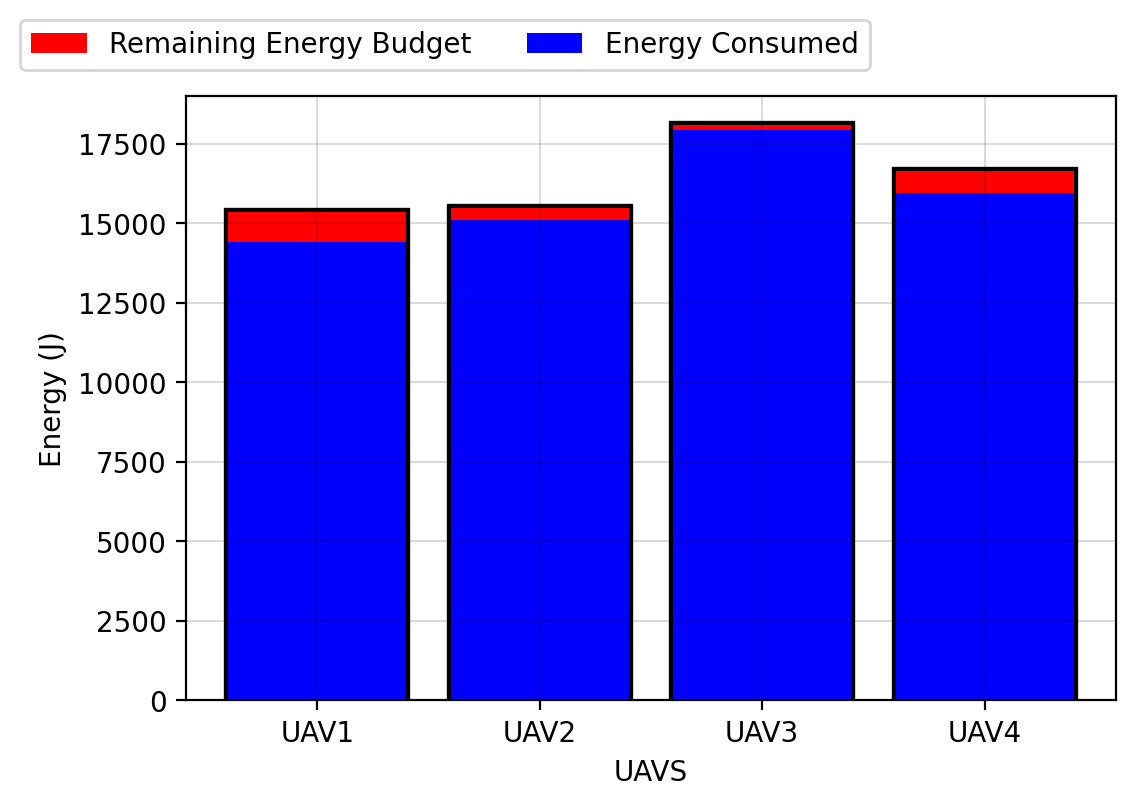}
    \caption{Energy consumption by each UAV}
    \label{fig:energyconsumption}
    \endminipage\hfill
\minipage{0.24\textwidth}%
\includegraphics[width=1\linewidth]{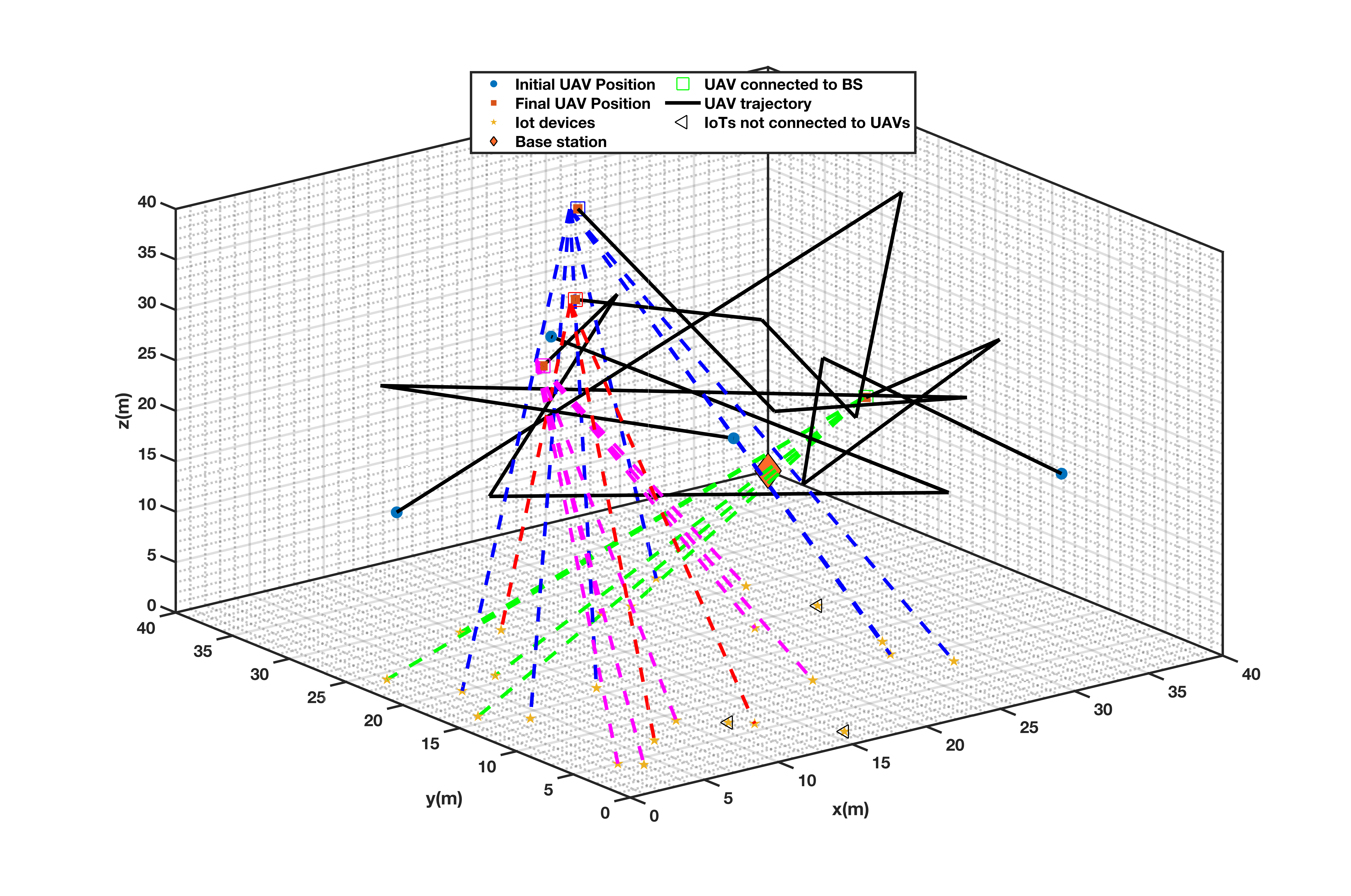}
        \caption{UAV trajectory and association for $k=5$}
        \label{UAV_assoc_k_4}
\endminipage
\end{figure*}

Now that we have succeeded in getting a convex reformulation of $\mathcal{P}$, we use convex optimization methods to solve it efficiently, namely the interior-point method. 

After obtaining the optimal probabilities, in practice at each time interval $k$, we use the  procedure \textbf{U2B} to decide the association between the UAVs and the BS, and \textbf{I2U} to decide the association of IoT devices with the UAVs.

\textbf{U2B procedure}: Let $\textbf{b}^*$ be the optimal UAV-BS association vector obtained by solving the convex optimization. To link a given UAV $u$ to the BS during time interval $k$, we draw a random number between $0$ and $1$. If this number is lower than the  obtained optimal association probability $b_{u}[k]^*$, UAV $u$ is associated with the BS during time interval $k$ and therefore is able to transmit the collected updates to the BS.

\textbf{I2U procedure}: Let $\textbf{a}^*$ be the optimal IoT device-UAV association matrix obtained by solving the convex optimization problem. We assume that during a time interval $k$, a device $i$ is allowed to send its updates to at most one UAV. The question that arises here is how to select at most one UAV according to the association probability vector $\textbf{a}^*$? 
For ease of notations, we remove the dependency on $k$. The probability selection vector of device $i$  with respect to the UAVs is given by $\textbf{a}_i^*=(a_{i,1}^*,\dots,a_{i,U}^*)$. To pick at most one UAV according to the probability vector $\textbf{a}_i^*$, we construct a virtual set of UAVs, where each UAV is represented according to its probability of selection. For example, let $U=3$ and  $\textbf{a}_i^*=(0.3,0.2,0.1)$. To select its associated UAV, device $i$ has to pick uniformly a random one UAV from the set $\{\text{UAV}_1,\text{UAV}_1,\text{UAV}_1,\text{UAV}_2,\text{UAV}_2,\text{UAV}_3,\text{UAV}_0,\text{UAV}_0,\\ \text{UAV}_0,\text{UAV}_0\}$ with $\text{UAV}_0$ is a hypothetical UAV which simply indicates that device $i$ will not be associated during that time interval.


Algorithm \ref{alg:1} summarizes  the different steps of our approach.
\begin{algorithm}[h]
\caption{Joint probability Association and position Scheduling (JAS)}\label{alg:1}
\begin{algorithmic}[1]
\State{\textbf{Input: } The maximum energy budget: $E_{u}^{max} \, ~ \forall u \in \mathcal{U}$ , the minimum IoT rate: $\bar{R_{I}}$, the minimum UAV rate: $\bar{R_{U}}$, the minimum distance between UAV: $d_{min}$. $K$ number of communication rounds.}
\State{\textbf{Optimization:} }
{Solve the convex reformulation of Problem $\mathcal{P}$ using a convex programming method, such as the interior-point method, to obtain the optimal association probabilities $\mathbf{a}^*$ and $\mathbf{b}^*$; and the optimal UAVs' positions over time.
}\label{Line1}
\State{\textbf{Association:} }
\While{$k \le K$}
    \State{Associate IoT devices to UAVs using  \textbf{I2U procedure}}\label{Line3}
    \State{Connect UAVs to BS using \textbf{U2B procedure} }\label{Line5}
    \State{$k=k+1$}
\EndWhile
\end{algorithmic}
\end{algorithm}
\section{Simulation Results}\label{Simu}

To evaluate the performance of our approach, we  consider the following scenario. We assume $K=10$, the number of UAVs $U=4$, the number of IoT devices  $I=25$, the speed to $v=15m/s$, the minimum distance between two UAVs $d_{min}=15m$, and a 3D area of $100\times 100 \times 100m^3$. Initially, the UAVs are randomly scattered in the 3D space.
We compare the proposed approach with a  deterministic approach similar to the one proposed in~\cite{Yang2019}.
The results of this comparison are shown in Fig.~ \ref{fig:comp}. As one may expect, our method gives a smaller value of AoU when compared to the deterministic approach, which in turn is better than a random feasible solution of the optimization problem. 

In Fig.~\ref{fig:funcplot}, we plot the objective value over iterations. As the plot shows, the value of the function decreases over the iterations until it reaches the optimum. 

Fig.~\ref{fig:energyconsumption} shows the energy consumed by each UAV against their maximum energy. We can see that none of the UAVs exceeded their energy limits. 

In Fig.~\ref{UAV_assoc_k_4}, we plot the UAVs trajectory and the associations for the last time interval $k=5$. In this figure, we show the 3D positioning of the UAVs as well as the trajectory followed by each UAV from $k=1$ to $k=5$. We also represent the associations between each UAV and the IoT devices as well as the UAVs selected by the BS for transmission during the last time interval. 
We can notice that, the UAVs try to position themselves in order to communicate well with the IoT devices and the BS. We can also see that most of the IoT devices are served by the UAVs.

    



\section{Conclusion}\label{Conc}
In this paper, we consider the optimization of the expected AoU in a UAV-assisted network. Our objective is to find the optimal 3D location scheduling and the probabilities of association between the IoT devices, the UAVs, and the BS. We first formulate the problem as a non-linear programming. Then, we propose a convex reformulation of the studied optimization. We solve the convex optimization using an interior-point based algorithm. Finally, we show that the proposed approach outperforms the benchmark algorithm.
\section*{Acknowledgment}
Research reported in this work was supported by the European Commission in partnership with the African Academy of Sciences (AAS) through ARISE-PP Grant (Grant N°:ARISE-PP-FA-109).

\balance

\bibliographystyle{IEEEtran}
\bibliography{main}

\end{document}